\documentclass[12pt]{amsart}

\usepackage{amsmath,amssymb,amsfonts,amsthm,latexsym,graphicx,multirow,hyperref,enumerate}

     \newcommand\Cay{\mathrm{Cay}}

\newcommand\K{\mathsf{K}}

\newcommand\Nor{\mathbf{N}}

\newtheorem{theorem}{Theorem}[section]
\newtheorem{lemma}[theorem]{Lemma}
\newtheorem{proposition}[theorem]{Proposition}
\newtheorem{corollary}[theorem]{Corollary}

\theoremstyle{definition}

\newtheorem*{remark}{Remark}

\begin{document}

\title[Subgroup perfect codes]{Characterization of subgroup perfect codes in Cayley graphs}

\author[Chen]{Jiyong Chen}
\address{Department of Mathematics\\Southern University of Science and Technology\\Shenzhen, Guangdong 518055\\People's Republic of China}
\email{chenjy@sustech.edu.cn}

\author[Wang]{Yanpeng Wang}
\address{School of Mathematical Sciences\\Peking University\\Beijing 100871\\People's Republic of China;
Also affiliated with Rongcheng Campus\\Harbin University of Science and Technology\\Harbin, Heilongjiang 150080\\People's Republic of China}
\email{wangyanpeng@pku.edu.cn}

\author[Xia]{Binzhou Xia}
\address{School of Mathematics and Statistics\\The University of Melbourne\\Parkville, VIC 3010\\Australia}
\email{binzhoux@unimelb.edu.au}

\maketitle

\begin{abstract}
A subset $C$ of the vertex set of a graph $\Gamma$ is called a perfect code in $\Gamma$ if every vertex of $\Gamma$ is at distance no more than $1$ to exactly one vertex of $C$. A subset $C$ of a group $G$ is called a perfect code of $G$ if $C$ is a perfect code in some Cayley graph of $G$. In this paper we give sufficient and necessary conditions for a subgroup $H$ of a finite group $G$ to be a perfect code of $G$. Based on this, we determine the finite groups that have no nontrivial subgroup as a perfect code, which answers a question by Ma, Walls, Wang and Zhou.


\end{abstract}

\section{Introduction}

In the paper, all groups considered are finite, and all graphs considered are finite, undirected and simple.
A subset $C$ of the vertex set of a graph $\Gamma$ is called a \emph{perfect code}~\cite{Kratochvil1986} in $\Gamma$ if every vertex of $\Gamma$ is at distance no more than $1$ to exactly one vertex of $C$ (in particular, $C$ is an independent set of $\Gamma$).
In some references, a perfect code is also called an \emph{efficient dominating set}~\cite{DS2003} or \emph{independent perfect dominating set}~\cite{Lee2001}.

In the study of perfect codes in graphs, special attention has been paid to perfect codes in Cayley graphs~\cite{Deng2014,DS2003,DSLW2017,FHZ2017,OPR2007,KM2013,CM2013}, which generalize perfect codes in many classical settings including Hamming codes and Lee codes.
Denote by $e$ the identity element of any group.
For a group $G$ and an inverse-closed subset $S$ of $G\setminus\{e\}$, the \emph{Cayley graph} $\Cay(G,S)$ on $G$ with \emph{connection set} $S$ is the graph with vertex set $G$ such that $x,y\in G$ are adjacent if and only if $yx^{-1}\in S$.
If a subset $C$ of $G$ is a perfect code in some Cayley graph of $G$, then $C$ is called a \emph{perfect code of $G$}.

The problem whether a subgroup is a perfect code of the group has attracted notable attention.
Given a normal subgroup $H$ of a group $G$, a sufficient and necessary condition for $H$ to be a perfect code of $G$ was given in~\cite{HXZ2018} as follows (we have replaced $g$ and $gh$ in the statement of~\cite[Theorem~2.2]{HXZ2018} with $x$ and $y$ respectively).

\begin{proposition}{\cite[Theorem~2.2]{HXZ2018}}\label{prop1}
Let $G$ be a group and $H$ be a normal subgroup of $G$. Then $H$ is a perfect code of $G$ if and only if for each $x\in G$ such that $x^2\in H$, there exists $y\in xH$ such that $y^2=e$.
\end{proposition}

As an application of this result, the subgroup perfect codes of cyclic groups were completely determined in~\cite{HXZ2018}.
It turns out that a subgroup $H$ of a cyclic group $G$ is a perfect code of $G$ if and only if either $|H|$ or $|G/H|$ is odd (see~\cite[Corollary~2.8]{HXZ2018}).
The authors in~\cite{HXZ2018} also classified the subgroup perfect codes of dihedral groups (see~\cite[Theorem~2.11]{HXZ2018}).
Recently, the subgroup perfect codes of generalized quaternion groups were classified in~\cite{MWWZ} (see~\cite[Theorem~1.7]{MWWZ}).
Recall that a generalized quaternion group is a group of order $4n$ for some integer $n\geqslant 2$ with presentation
\[
\langle a,b\mid a^{2n}=e,b^2=a^n,a^b=a^{-1}\rangle.
\]

In this paper, we give a sufficient and necessary condition for a general subgroup $H$ of a group $G$ to be a perfect code of $G$. Note that for a subgroup $H$ of a group $G$, an inverse-closed right transversal of $H$ in $G$ is also a left transversal of $H$ in $G$.

\begin{theorem}\label{thm2}
Let $G$ be a group and let $H$ be a subgroup of $G$.
Then the following are equivalent:
\begin{enumerate}[{\rm(a)}]
\item $H$ is a perfect code of $G$;
\item there exists an inverse-closed right transversal of $H$ in $G$;
\item for each $x\in G$ such that $x^2\in H$ and $|H|/|H\cap H^x|$ is odd, there exists $y\in Hx$ such that $y^2=e$;
\item for each $x\in G$ such that $HxH=Hx^{-1}H$ and $|H|/|H\cap H^x|$ is odd, there exists $y\in Hx$ such that $y^2=e$.
\end{enumerate}
\end{theorem}

\begin{remark}
The equivalence between~(a) and~(b) in Theorem~\ref{thm2} has already been known (see~\cite[Lemma~2.1]{HXZ2018} and~\cite[Lemma~2.2]{MWWZ}). Also note that if $H$ is a normal subgroup of $G$ then $|H|/|H\cap H^x|=1$ and $Hx=xH$ for each $x\in G$. Thus the sufficient and necessary condition in Proposition~\ref{prop1} is statement~(c) of Theorem~\ref{thm2} in the special case when $H$ is normal in $G$.
\end{remark}

A group $G$ is said to be \emph{code-perfect} if every subgroup of $G$ is a perfect code of $G$. Ma, Walls, Wang and Zhou proved that a group $G$ is code-perfect if and only if it has no elements of order four~\cite[Theorem~1.1]{MWWZ}.
As a natural question in the counterpart, they also proposed to study which groups have no nontrivial proper subgroup perfect code (see the paragraph after Theorem~1.5 of \cite{MWWZ}) and proved that cyclic $2$-groups are examples of such groups~\cite[Theorem~1.6]{MWWZ}.
Another (trivial) family of examples are groups of prime order as they have no nontrivial proper subgroup.
Our next main result completely answers the question by showing that a group of composite order has no nontrivial proper subgroup as a perfect code if and only if it is a cyclic $2$-group or a generalized quaternion $2$-group.

\begin{theorem}\label{thm5}
Let $G$ be a group of composite order.
Then the following are equivalent:
\begin{enumerate}[{\rm(a)}]
\item $G$ has no nontrivial proper subgroup as a perfect code;
\item $G$ is a $2$-group with a unique involution;
\item $G$ is either a cyclic $2$-group or a generalized quaternion $2$-group.
\end{enumerate}
\end{theorem}

The rest of this paper is organized as follows. In Section~\ref{sec2} we prove Theorem~\ref{thm2} and list some of its corollaries. Based on this, we then prove Theorem~\ref{thm5} in section~\ref{sec4}.

\section{Proof of Theorem~\ref{thm2}}\label{sec2}

Throughout this section, let $G$ be a group and let $H$ be a subgroup of $G$.
Denote by $[G{:}H]$ the set of right cosets of $H$ in $G$, and define a binary relation $\sim$ on $[G{:}H]$ such that
\[
Hx\sim Hy\Leftrightarrow y\in Hx^{-1}H.
\]
It is readily seen that the relation $\sim$ is well-defined and symmetric.
Now let $\Gamma$ be the graph with vertex set $[G{:}H]$ such that $\{Hx,Hy\}$ is an edge if and only if $Hx\sim Hy$ and $Hx\neq Hy$.
Then for each $x\in G$ and $h\in H$,
\begin{equation}\label{eq1}
Hxh\sim Hy\text{ for all $y\in x^{-1}H$}\quad\text{and}\quad Hx^{-1}h\sim Hz\text{ for all $z\in xH$}.
\end{equation}
Since each double coset of $H$ is a union of right cosets of $H$ in $G$, we may view $HxH$ and $Hx^{-1}H$ as sets of vertices of $\Gamma$. Then~\eqref{eq1} shows that the induced subgraph by $HxH\cup Hx^{-1}H$ is a connected component of $\Gamma$, which we denote by $\Gamma_x$.

The next lemma describes the connected component $\Gamma_x$ of $\Gamma$.

\begin{lemma}\label{lem4}
Let $\Gamma_x$ be as above and let $m=|H|/|H\cap H^x|$. If $HxH=Hx^{-1}H$, then $\Gamma_x$ is the complete graph $\K_m$. If $HxH\neq Hx^{-1}H$, then $\Gamma_x$ is the complete bipartite graph $\K_{m,m}$. 
\end{lemma}

\begin{proof}
Since $|HxH|/|H|=|Hx^{-1}H|/|H|=m$, it follows from our convention that both $HxH$ and $Hx^{-1}H$ are regarded as a set of $m$ vertices of $\Gamma$. If $HxH=Hx^{-1}H$, then $\Gamma_x$ has vertex set $HxH$, and we see from~\eqref{eq1} that each vertex in $HxH$ is adjacent to all the other vertices in $HxH$. This shows that the condition $HxH=Hx^{-1}H$ implies $\Gamma_x=\K_m$. If $HxH=Hx^{-1}H$, then $\Gamma_x$ has vertex set $HxH\cup Hx^{-1}H$ with $HxH\cap Hx^{-1}H=\emptyset$, and~\eqref{eq1} shows that each vertex in $HxH$ is adjacent to all the vertices in $Hx^{-1}H$. Hence, under the condition $HxH\neq Hx^{-1}H$, we have $\Gamma_x=\K_{m,m}$ with parts $HxH$ and $Hx^{-1}H$. This completes the proof.
\end{proof}

The graph $\Gamma$ defined above plays a key role in the following proof of Theorem~\ref{thm2}.

\begin{proof}[Proof of Theorem~$\ref{thm2}$]
It follows from~\cite[Lemma~2.1]{HXZ2018} that (a)$\Leftrightarrow$(b).

To prove (b)$\Rightarrow$(c), suppose that~(b) holds, which means that there exists an inverse-closed right transversal $S$ of $H$ in $G$.
Let $x\in G$ such that $x^2\in H$ and $|H|/|H\cap H^x|$ is odd.
Define a graph $\Sigma$ with vertex set $[G{:}H]$ and edge set
\[
\{\{Hs,Hs^{-1}\}\mid s\in S,\, s^2\neq e\}.
\]
Then $\Sigma$ is a spanning subgraph of $\Gamma$, and each vertex of $\Sigma$ has valency $0$ or $1$.
Consider the subgraph $\Sigma_x$ of $\Sigma$ induced by the vertex set $V(\Gamma_x)$ of $\Gamma_x$.
It follows that each vertex of $\Sigma_x$ has valency $0$ or $1$.
Moreover, we deduce from $x^2\in H$ that $HxH=Hx^{-1}H$, and so Lemma~\ref{lem4} asserts that $\Gamma_x=\K_m$ with $m=|H|/|H\cap H^x|$.
As a consequence, $|V(\Sigma_x)|=|V(\Gamma_x)|=m$ is odd.
Therefore, $\Sigma_x$ has an isolated vertex, say, $Hz$ with $z\in S$.
Since $\Sigma_x$ is the subgraph of $\Sigma$ induced by $V(\Gamma_x)$ while $\Gamma_x$ is a connected component of $\Gamma$, it follows that $Hz$ is an isolated vertex of $\Sigma$, which means that $z^2=e$.
As $Hz\in V(\Gamma_x)$, we have $z\in HxH$ and hence $z=h_1xh_2$ for some $h_1,h_2\in H$.
Take $y=h_2zh_2^{-1}=h_2h_1x\in Hx$.
Then we deduce from $z^2=e$ that $y^2=e$.
This shows that~(c) holds.

Next we prove (c)$\Rightarrow$(d).
Suppose that~(c) holds.
Let $x\in G$ such that $HxH=Hx^{-1}H$ and $|H|/|H\cap H^x|$ is odd.
Then $x^{-1}\in HxH$, which means that there exist $h_1,h_2\in H$ with $x^{-1}=h_1xh_2$.
This yields $(h_1x)^2=h_1h_2^{-1}$ and so $(h_1x)^2\in H$.
Moreover, since $H^{h_1x}=H^x$, we see that $|H|/|H\cap H^{h_1x}|=|H|/|H\cap H^x|$ is odd.
Hence statement~(c) implies that there exists $y\in Hh_1x$ with $y^2=e$.
Note that $Hh_1x=Hx$.
We then conclude that~(d) holds.

Now we embark on the proof of (d)$\Rightarrow$(b).
Suppose that~(d) holds. Let $Hx_1,\dots,Hx_k$ be the vertices of $\Gamma$ such that $\Gamma_{x_i}$ is a complete graph of odd order.
Then we derive from Lemma~\ref{lem4} that $Hx_iH=Hx_i^{-1}H$ with $|H|/|H\cap H^{x_i}|$ odd for each $i\in\{1,\dots,k\}$, and each connected component of $\Gamma$ other than $\Gamma_{x_1},\dots,\Gamma_{x_k}$ is either a complete graph of even order or a complete bipartite graph.
For each $i\in\{1,\dots,k\}$, statement~(d) asserts that there exists $y_i\in Hx_i$ with $y_i^2=e$. In particular, $Hy_i=Hx_i$ is a vertex of $\Gamma_{x_i}$, a connected component of $\Gamma$ that is a complete graph of odd order.
It follows that $\Gamma$ has a matching with vertex set $V(\Gamma)\setminus\{Hy_1,\dots,Hy_k\}$.
Let $\{Hu_1,Hv_1\},\dots,\{Hu_\ell,Hv_\ell\}$ be the edges of this matching.
By the adjacency in $\Gamma$, for each $j\in\{1,\dots,\ell\}$, we have $v_j=g_ju_j^{-1}h_j$ for some $g_j,h_j\in H$.
Let $z_j=h_j^{-1}u_j$.
Then $Hz_j=Hu_j$ and $Hz_j^{-1}=Hu_j^{-1}h_j=Hg_j^{-1}v_j=Hv_j$.
Hence
\[
\{Hz_1,Hz_1^{-1},\dots,Hz_\ell,Hz_\ell^{-1}\}=\{Hu_1,Hv_1,\dots,Hu_\ell,Hv_\ell\}=V(\Gamma)\setminus\{Hy_1,\dots,Hy_k\}
\]
and so
\[
[G{:}H]=V(\Gamma)=\{Hy_1,\dots,Hy_k,Hz_1,Hz_1^{-1},\dots,Hz_\ell,Hz_\ell^{-1}\}.
\]
Since $y_i^2=e$, the set $\{y_1,\dots,y_k,z_1,z_1^{-1},\dots,z_\ell,z_\ell^{-1}\}$ is inverse-closed.
This shows that there exists an inverse-closed right transversal of $H$ in $G$, proving~(b).
\end{proof}

As a first corollary of Theorem~\ref{thm2}, we verify a result in~\cite{MWWZ} stating that every group of odd order is code-perfect.

\begin{corollary}\label{cor1}{\cite[Corollary~1.2]{MWWZ}}
Let $G$ be a group of odd order and let $H$ be a subgroup of $G$.
Then $H$ is a perfect code of $G$.
\end{corollary}

\begin{proof}
Let $x\in G$ such that $x^2\in H$.
Since $|G|$ is odd, the element $x$ has odd order.
Thus the condition $x^2\in H$ implies that $x\in H$.
Take $y=e$.
Then $y\in H=Hx$ and $y^2=e$.
This shows that statement~(c) of Theorem~\ref{thm2} holds, and so $H$ is a perfect code of $G$.
\end{proof}

Note in Theorem~\ref{thm2} that if $|H|$ is a power of $2$, then $|H|/|H\cap H^x|$ is odd if and only if $x\in\Nor_G(H)$.
We then obtain the following corollary of Theorem~\ref{thm2}.

\begin{corollary}\label{cor2}
Let $G$ be a group and let $H$ be a $2$-subgroup of $G$.
Then the following are equivalent:
\begin{enumerate}[{\rm(a)}]
\item $H$ is a perfect code of $G$;
\item for each $x\in\Nor_G(H)$ such that $x^2\in H$, there exists $y\in Hx$ such that $y^2=e$;
\item for each $x\in\Nor_G(H)\setminus H$ such that $x^2\in H$, there exists an involution in $Hx$.
\end{enumerate}
\end{corollary}

A further corollary as follows says that every Sylow $2$-subgroup is a perfect code.

\begin{corollary}\label{cor5}
Let $G$ be a group and let $H$ be a Sylow $2$-subgroup of $G$.
Then $H$ is a perfect code of $G$.
\end{corollary}

\begin{proof}
Let $H$ be a Sylow $2$-subgroup of $G$.
Suppose that there exists $x\in\Nor_G(H)\setminus H$ such that $x^2\in H$.
Then $\langle H,x\rangle=H\cup Hx$ with $H\cap Hx=\emptyset$.
However, this implies $|\langle H,x\rangle|=2|H|$, contradicting the condition that $H$ is a Sylow $2$-subgroup of $G$.
Hence there is no $x\in\Nor_G(H)\setminus H$ such that $x^2\in H$.
Thereby statement~(c) of Corollary~\ref{cor2} holds, and so $H$ is a perfect code of $G$.
\end{proof}

\section{Proof of Theorem~\ref{thm5}}\label{sec4}

It is well-known that if $G$ is a $2$-group with a unique involution, then $G$ is either a cyclic group or a generalized quaternion group (see~\cite[105~Theorem~VI]{Burnside1911}).
Conversely, if $G$ is a cyclic $2$-group or a generalized quaternion group $2$-group, then $G$ has a unique involution.
This is obvious if $G$ is cyclic, and is true if 
\[
G=\langle a,b\mid a^{2n}=e,b^2=a^n,a^b=a^{-1}\rangle
\] 
is a generalized quaternion group since the elements outside the cyclic subgroup $\langle a\rangle$ all have order $4$.
Thus we have the following lemma.

\begin{lemma}\label{lem2}
Let $G$ be a group. Then $G$ is a $2$-group with a unique involution if and only if $G$ is either a cyclic $2$-group or a generalized quaternion $2$-group.
\end{lemma}

We are now in a position to prove Theorem~\ref{thm5}.

\begin{proof}[Proof of Theorem~$\ref{thm5}$]
The equivalence of~(b) and~(c) follows from Lemma~\ref{lem2}.
It remains to show that (a)$\Leftrightarrow$(b).
This is obvious if $|G|=2$.
Thus we assume $|G|>2$ in the rest of the proof.

First suppose that~(a) holds, that is, $G$ has no nontrivial proper subgroup as a perfect code.
Then Corollary~\ref{cor1} implies that $|G|$ is even.
Let $P$ be a Sylow~$2$-subgroup of $G$.
Then $P$ is nontrivial, and Corollary~\ref{cor5} asserts that $P$ is a perfect code of $G$.
This leads to $P=G$, which means that $G$ is a $2$-group.
We then derive from Corollary~\ref{cor2} that for each nontrivial proper subgroup $H$ of $G$, there exists $x\in\Nor_G(H)\setminus H$ with $x^2\in H$ such that $Hx$ has no involution.
Let $|G|=2^m$, where $m\geqslant2$ is an integer.
Take $H_1$ to be any subgroup of order $2$ in $G$.
Suppose we have a subgroup $H_i$ of order $2^i$ in $G$ for some $i\in\{1,\dots,m-1\}$ such that $H_i$ has a unique involution.
Then there exists $x_i\in\Nor_G(H_i)\setminus H_i$ with $x_i^2\in H_i$ such that $H_ix_i$ has no involution.
Now take $H_{i+1}=\langle H_i,x_i\rangle$.
It follows that $H_{i+1}=H_i\cup H_ix_i$ is a group of order $2|H_i|=2^{i+1}$, and each involution of $H_{i+1}$ lies in $H_i$ and so is a unique involution of $H_{i+1}$.
By induction, we then obtain a subgroup $H_m$ of order $2^m$ in $G$ such that $H_m$ has a unique involution.
Since $|H_m|=2^m=|G|$, this implies that $G=H_m$ has a unique involution, as~(b) states.

Next suppose that~(b) holds, that is, $G$ is a $2$-group with a unique involution.
Let $H$ be a nontrivial proper subgroup of $G$.
Then there exists a subgroup $K$ of $G$ such that $H<K$ and $|K|=2|H|$.
In particular, $H$ is normal in $K$.
Take $x\in K\setminus H$.
Then $K=H\cup Hx$ with $H\cap Hx=\emptyset$, and $x\in\Nor_G(H)\setminus H$ with $g^2\in H$.
Moreover, since $H$ has an involution and the involution of $G$ is unique, we infer that there is no involution in $Hx$.
Hence by Corollary~\ref{cor2}, $H$ is not a perfect code of $G$.
This shows that $G$ has no nontrivial proper subgroup as a perfect code, as~(a) states, which completes the proof.
\end{proof}

\vskip0.1in
\noindent\textsc{Acknowledgement.} The second author gratefully acknowledges the financial support from China Scholarship Council no.~201806010040. The authors would like to thank Prof.~Sanming Zhou for bringing~\cite{MWWZ} into their attention and the anonymous referee for helpful suggestions.


\begin{thebibliography}{}

\bibitem{Burnside1911}
W. Burnside,
\textit{Theory of groups of finite order (2nd edition)}, Cambridge University Press, Cambndge, 1911,

\bibitem{DS2003}
I. J. Dejter and O. Serra, Efficient dominating sets in Cayley graphs,
\textit{Discrete Appl. Math.}, 129 (2003), no. 2--3, 319--328.

\bibitem{Deng2014}
Y.-P. Deng, Efficient dominating sets in circulant graphs with domination number prime,
\textit{Inform. Process. Lett.}, 114 (2014), no. 12, 700--702.

\bibitem{DSLW2017}
Y.-P. Deng, Y.-Q. Sun, Q. Liu and H.-C. Wang, Efficient dominating sets in circulant
graphs,
\textit{Discrete Math.}, 340 (2017), no. 7, 1503--1507.

\bibitem{FHZ2017}
R. Feng, H. Huang and S. Zhou, Perfect codes in circulant graphs,
\textit{Discrete Math.}, 340 (2017), no. 7, 1522--1527.

\bibitem{HXZ2018}
H. Huang, B. Xia and S. Zhou, Perfect codes in Cayley graphs,
\textit{SIAM J. Discrete Math.}, 32 (2018), no. 1, 548--559.

\bibitem{Kratochvil1986}
J. Kratochv\'{i}l, Perfect codes over graphs, \textit{J. Combin. Theory Ser. B}, 40 (1986), no. 2, 224--228.

\bibitem{Lee2001}
J. Lee, Independent perfect domination sets in Cayley graphs, \textit{J. Graph Theory}, 37 (2001), no. 4, 213--219.

\bibitem{MWWZ}
X. Ma, G. L. Walls, K. Wang and S. Zhou, Subgroup perfect codes in Cayley graphs, https://arxiv.org/abs/1904.01858.

\bibitem{OPR2007}
N. Obradovi\'{c}, J. Peters and G. Ru\v{z}i\'{c}, Efficient domination in circulant graphs with two chord lengths,
\textit{Inform. Process. Lett.}, 102 (2007), no. 6, 253--258.

\bibitem{KM2013}
K. Reji Kumar and G. MacGillivray, Efficient domination in circulant graphs,
\textit{Discrete Math.}, 313 (2013), no. 6, 767--771.

\bibitem{CM2013}
T. Tamizh Chelvam and S. Mutharasu, Subgroups as efficient dominating sets in Cayley graphs,
\textit{Discrete Appl. Math.}, 161 (2013), no. 9, 1187--1190.

\bibliographystyle{100}
\end{thebibliography}
\end{document}